\documentclass[reqno]{amsart}
\usepackage{float}
\usepackage{tikz}
\usepackage{amsmath} 
\usepackage{graphicx} 
\usepackage{latexsym}
\usepackage{amsfonts}
\usepackage{amssymb}
\usepackage{color} 
\usepackage[normalem]{ulem}

\theoremstyle{plain}
\newtheorem{theorem}{Theorem}
\newtheorem{corollary}[theorem]{Corollary}
\newtheorem{lemma}{Lemma}
\newtheorem{remark}{Remark}
\newtheorem{proposition}[theorem]{Proposition}
\theoremstyle{definition}

\newtheorem*{remark*}{Remark}

\newcommand{\rd}{\mathbb R^d}
\renewcommand{\Pr}{\mathbf P}
\newcommand{\E}{\mathbf E}
\renewcommand{\H}{\mathbb H}
\newcommand{\ind}[1]{{\rm \bf1}\{#1\}}

\title{Harmonic measure in a multidimensional gambler's problem}
\thanks{
        D. Denisov was supported by a Leverhulme Trust Research Project Grant  RPG-2021-105. 
        V. Wachtel was partially supported by DFG
    }
 \author[Denisov]{Denis Denisov}
 \address{Department of Mathematics, University of Manchester, Oxford Road, Manchester M13 9PL, UK}
 \email{denis.denisov@manchester.ac.uk}
 
\author[Wachtel]{Vitali Wachtel}
\address{Faculty of Mathematics, Bielefeld University, Germany}
\email{wachtel@math.uni-bielefeld.de}
 
\date{}

\begin{document}

\begin{abstract}
    We consider a random walk in a truncated cone $K_N$, 
    which is obtained by 
     slicing cone $K$ 
    by a hyperplane at a growing level of order $N$. 
    We study the behaviour of the Green function 
    in this truncated cone as $N$ increases. 
    Using these results we also obtain the asymptotic 
    behaviour of the harmonic measure. 
    
    The obtained results are applied to a multidimensional gambler's problem 
    studied by Diaconis and Ethier (2022). 
    In particular we confirm their conjecture that 
    the probability of eliminating players in a particular order 
    has the same exact asymptotic behaviour as for the Brownian motion approximation. 
    We also provide a  rate of convergence 
    of this probability towards this approximation. 

\end{abstract}

\maketitle

\section{Introduction}
Consider a random walk $\{S(n),n\geq1\}$ on a lattice $R$ which is a linear transformation of the standard $d$-dimensional lattice 
$\mathbb  Z^d$, $d\geq1$, where
$$
S(n)=X(1)+\cdots+X(n)
$$
and $\{X(n), n\geq1\}$ is a family of independent copies of a random vector $X=(X_1,X_2,\ldots,X_d)$.
Denote by $\mathbb{S}^{d-1}$ the unit sphere of $\rd$ and $\Sigma$ an open and connected subset of
$\mathbb{S}^{d-1}$. Let $K$ be the cone generated by the rays emanating from the origin and passing
through $\Sigma$, i.e. $\Sigma=K\cap \mathbb{S}^{d-1}$.
Let $\tau$ be the exit time from $K$ of the random walk, that is
\[ 
\tau:=\inf\{n\ge 1: S(n)\notin K\}. 
\] 

In the series of papers \cite{DW10, DW15, DW19, DW22} we have studied
the tail behaviour of $\tau$ and have proved various conditional limit theorems for $S(n)$ conditioned on $\{\tau>n\}$. 

In the present paper we 
are going to study some properties of walks conditioned to stay in $K$ sliced by a hyperplane at a high level. 
Our attention to this problem was drawn by Persi Diaconis. 
In his joint paper~\cite{DE20} with Stewart Ethier, the following gambler's ruin problem with three players has been considered.
The players have starting capitals $A$, $B$ and $C$. 
At each step, a pair of players is chosen uniformly at random, 
the chosen players play a fair Heads-or-Tails game resulting in the transfer of one unit of capital 
between these players. 
The game is played until one of the players wins all  $A+B+C$ units of capital. 
Since the total capital of all players remains constant, it suffices to keep track of capitals of two players only. As a result, one can model this game by the following $2$-dimensional random walk. Let $Y(n)$ be independent random vectors with the uniform distribution on the set 
$$
\{(1,0), (-1,0), (0,1), (0,-1), (1,-1), (-1,1)\}.
$$
One of the players is eliminated when the random walk 
$$
Z(n)=Y(1)+Y(2)+\ldots+Y(n)
$$
starting at $(A,B,C)$ hits the boundary of the triangle $\{(x,y):\ x,y\ge0,\ x+y\le N\}$, where $N=A+B+C$. 
Using the results of~\cite{DHESC19} 
Diaconis and Ethier have shown that if $A=B=1$ then there exist constants $c_1$ and $c_2$ such that 
\begin{equation}
\label{eq:DE}
\frac{c_1}{N^3}\le\Pr(\text{third player  goes broke first})
\le \frac{c_2}{N^3}.
\end{equation}
(This means that $Z(n)$ exits the triangle via $\{(x,y):x+y=N\}$.)
The techniques of~\cite{DHESC19} 
give similar estimates for very general (inner uniform) domains.  
Diaconis and Ethier~\cite{DE20}
also obtained a number of numerical results 
which showed  a good fit of this probability with the 
corresponding asymptotics for the Brownian motion 
and strongly suggest 
there  exists a positive constant $c$ such that 
\begin{equation}\label{eq.asymp}
\Pr(\text{third player  goes broke first})
\sim\frac{c}{N^3}\quad\text{as }N\to\infty.
\end{equation}
 O'Connor and Saloff-Coste~\cite{OConnnorSaloffCoste22} 
have proven an analogue of \eqref{eq:DE} for the gambler's problem with four players.  
In this paper we will confirm that the lower and upper bounds 
can be strengthened and exact asymptotics hold, 
and, in particular,~\eqref{eq.asymp} and its analogues  are true.

The purpose of the present note is to consider similar exit problems for a  rather wide classes of cones and walks and to prove that the Brownian approximation is valid under mild moment conditions on walks.

Let $\H$ be a hyperplane satisfying the following conditions: 
\begin{itemize}
 \item[(i)] $0\notin\H$;
 \item[(ii)] $\H$ cuts the cone $K$ into two parts, one of these parts is bounded;
 \item[(iii)] there exists $\varepsilon>0$ such that the set 
 $\{x\in K:\,|x|<\varepsilon\}$ belongs to the bounded part.
\end{itemize}
Let $x_0\in\H$ be such that $\text{dist}(0,\H)=|x_0|$. For every $t>0$ we denote by $\H_t$ the hyperplane which is parallel to $\H$ and contains the point $tx_0$.

Define $\Gamma:=K\cap\H$. For every $A\subset\Gamma$ we denote by $K^{(A)}$ the cone consisting of all rays starting at zero and going through points of the set $A$. Now we may define 
$$
A_{s,t}:=K^{(A)}\cap\left(\bigcup_{u:s\le u<t}\H_u\right),
\quad 0\le s<t\le\infty.
$$
Clearly, $A_{s,t}$ is the part of $K^{(A)}$ between hyperplanes $\H_s$
and $\H_t$. 
We shall also use the notation
$$
K_{0,N}=K\cap\left(\bigcup_{u:u<N}\H_u\right),
\quad N>0.
$$

Now we can define the main objects of interest for us. Let  
\[
\sigma_N:=\inf\{n\ge 1\colon S(n)\notin K_{0,N}\}.     
\]
Relating to the gambler's problem described above, we see that the event 
'third player goes broke first' can be expressed in the following way:
$Z(\sigma_N)\in\Gamma_{N,\infty}$. We will be interested in the behaviour of the probabilities
$$
\Pr(S(\sigma_N)\in A_{N,\infty}),\ A\subset\Gamma 
\quad\text{and}\quad 
\Pr(S(\sigma_N)=y),\ y\in \Gamma_{N,\infty}.
$$
It is well-known that the distribution of the position at an exit time is closely related to the Green function of the corresponding killed process.  Let $G_N(x,y)$ be the Green function 
corresponding to the exit time $\sigma_N$, that is 
\[
    G_N(x,y):=\sum_{n=0}^\infty \Pr_x(S_n=y,\sigma_N>n). 
\]
For any set $A$ we also let 
\[
    G_N(x,A):=\sum_{y\in A} G_N(x,y). 
\]

In order to formulate our results we recall some known properties of random walks conditioned to stay in $K$. First, let $V(x)$ denote the positive harmonic function for $S(n)$ killed at leaving the cone $K$, i.e.,
$$
V(x)=\E_x[V(S(1));\tau>1],\quad x\in K.
$$
This function has been constructed in our recent paper \cite{DW19} under the following conditions:
\begin{itemize}
 \item[(i)]the cone $K$ is either convex or star-like and $C^2$,
 \item[(ii)] the random vector $X$ has zero mean, unit covariance matrix
 and moments $\E|X|^p$, $\E|X|^{2+\varepsilon}$ are finite.
\end{itemize}
The constant $p$ in the second assumption depends on the cone $K$ only and is equal to the degree of homogeneity of the positive solution to the classical Dirichlet problem in $K$. 
It can be found by solving an eigenvalue problem for a domain on a sphere, see~\cite{DW15} for further details.

Let 
$e_\H:=-\frac{x_0}{|x_0|}$ be the vector perpendicular to $\H$ and directed to $0$. 
We will require a one-dimensional projection of the random walk  
\[
S_{\H}(n) = (S(n), e_\H),\quad n\ge 0.    
\]
Let $v_\H(z)$ be the renewal function of the descending 
ladder height process of $(S_\H(n))_{n\ge 0}$. 

\begin{theorem}\label{thm:gN.global} 
    Assume that $K$ is either convex or starlike and $C^2$.
    Assume also that $X$ has zero mean, unit covariance matrix and 
    $\E|X|^{p\vee 2+\varepsilon}<\infty$. 
    Then, there exists $C$ such that 
    for any $A\subset \text{int}(\Gamma)$  
    \begin{equation}\label{eq:gN.upper.bound}
        G_N(x,A_{N-j,N-j+1}) 
        \le C \frac{V(x)(1+\text{ dist}(y,\Gamma_N))}{N^{p}}j, \quad  1\le j\le \frac N2. 
    \end{equation}   
    
    If $\E|X|^{p+d}<\infty$ 
    then there exists $C$ such that  
    \begin{equation}\label{eq:gn.local.bound}
        G_N(x,y)\le C
        \frac{V(x)}{N^{p+d-1}}j, 
        \quad y\in \Gamma_{N-j,N-j+1},\ j\le\frac{N}{2}. 
    \end{equation}    
    Furthermore, there exists a function $h$, depending on the cone $K$ only, such that, uniformly in $j=o(N)$,
    \begin{equation}\label{eq:gn.global.asymp}
        G_N(x,A_{N-j,N-j+1})\sim 
        \frac{V(x)}{N^{p+d-1}} 
        \sum_{y\in A_{N-j,N-j+1}} v_\H(\text{dist}(y,\Gamma_N))h(y/N), 
    \end{equation}    
    a closed form for the function $h$ is given in \eqref{eq:h-def}.
    
    Moreover, under the condition $\E|X|^{p+d}<\infty$,
    \begin{equation}\label{eq:gn.local.asymp}
    G_N(x,y)= 
        \frac{V(x)v_\H(\text{dist}(y,\Gamma_N))}{N^{p+d-1}} h(y/N)
        +o\left(\frac{\text{dist}(y,\Gamma_N)}{N^{p+d-1}}\right),
    \end{equation}    
    uniformly in $y\in K_{0,N}$ with $\text{dist}(y,\Gamma_N)=o(N)$. 
    \end{theorem}    
This result can be used to answer some questions on the gambler's problem mentioned above. It is easy to see that the random walk $Z(n)$ does not satisfy the conditions of Theorem~\ref{thm:gN.global}: 
$\E Y_1^2=\E Y_2^2=\frac{2}{3}$ and $\E Y_1Y_2=-\frac{1}{3}$. To adjust this walk we apply the linear transformation given by the matrix
$$
T=
\left(
\begin{array}{cr}
\sqrt{2} & \frac{\sqrt{2}}{2}\\
0 &\frac{\sqrt{6}}{2}
\end{array}
\right).
$$
The random walk $S(n)=TZ(n)$ satisfies the conditions of our theorem.
The cone $\mathbb{R}_+^2$ transforms under this mapping into the wedge $K$ with the opening angle $\frac{\pi}{3}$. Furthermore, the original cutting line transfers under $T$ to the line 
$\{(x_1,x_2): \sqrt{3}x_1+x_2=\sqrt{6}\}$. It is easy to see that the harmonic function for the Brownian motion in the wedge $K$ is given by 
$u(x)=3x_1^2x_2-x_2^3$. In particular, $p=3$ in this case. 
Furthermore, simple calculations show that the function $u$ is harmonic also for the discrete time walk $S(n)$ killed at leaving $K$. Thus, $V(x)=u(x)$ for all $x$ in the set $K\cap (T\mathbb{Z}^2)$. Since the random walk $S(n)$
can not jump over the boundary
$\Gamma_N=\{x: x_2+\sqrt{3}x_1=\sqrt{6}N\}$, every point $y\in\Gamma_N$
can be reached from $y^{(1)}:=y-(\sqrt{2},0)$ and from
$y^{(2)}:=y-(\frac{\sqrt{2}}{2},\frac{\sqrt{6}}{2})$. This implies that
$$
\Pr_x(S(\sigma_N)=y)=\frac{1}{6}G_n(x,y^{(1)})+\frac{1}{6}G_n(x,y^{(2)}).
$$

In order to apply \eqref{eq:gn.local.asymp}, it remains to calculate  the renewal function $v_{\H}$. It is rather obvious that the increments of $S_{\H}(n)$ take values $-\frac{\sqrt{6}}{2},0,\frac{\sqrt{6}}{2}$ with equal probabilities. Therefore, $v_{\H}(r)=r$ for every
$r\in\frac{\sqrt{6}}{2}\mathbb{N}$. In particular,
$$
v_{\H}(\text{dist}(y^{(1)},\Gamma_N))
=v_{\H}(\text{dist}(y^{(2)},\Gamma_N))=\frac{\sqrt{6}}{2}.
$$
Plugging this into \eqref{eq:gn.local.asymp}, we obtain
\begin{equation}
\label{eq:DE_2}
\Pr_x(S(\sigma_N)=y)=
\frac{1}{\sqrt{6}}h\left(\frac{y}{N}\right)\frac{V(x)}{N^4}
+o\left(\frac{1}{N^4}\right).
\end{equation}

If the game starts with capitals $A, B$ and $N-A-B$ then the random walk
$S(n)$ starts at the point $x=(\sqrt{2}A+\frac{\sqrt{2}}{2}B,\frac{\sqrt{6}}{2}B)$. Recalling that $V(x)=3x_1^2x_2-x_2^3$, one infers easily that 
$$
V(x)=3\sqrt{6} AB(A+B)
\quad\text{for }x=\left(\sqrt{2}A+\frac{\sqrt{2}}{2}B,
\frac{\sqrt{6}}{2}B\right).
$$

Summing over $y$ asymptotics in \eqref{eq:DE_2}, we conclude that the probability, 
that the third player (the player with the starting capital $N-A-B$) goes bankrupt first, is asymptotically equivalent to 
$$
3\frac{AB(A+B)}{N^{3}}\int_\Gamma h(z)dz.
$$
Furthermore, the probability that the first player wins the whole game is asymptotically equivalent to 
\begin{equation}\label{eq:asymp.3.rw} 
3\frac{AB(A+B)}{N^{3}}
\int_\Gamma \left(\sqrt{2}z_1-1\right)h(z)dz.
\end{equation}
As the same arguments apply to the standard Brownian motion 
we obtain that the same asymptotics with the same constant should hold, see 
Corollary~\ref{cor.rw} below. 

Let us now comment on the global asymptotics in 
\eqref{eq:gn.global.asymp}. The right hand side is just the sum of local asymptotics from \eqref{eq:gn.local.asymp}, the only advantage is the fact that for \eqref{eq:gn.global.asymp} we need a weaker moment assumption. 
Such a cumbersome expression is caused by the fact that we know only asymptotics for the renewal function $v_\H$. By the renewal theorem, $v_\H(r)\sim c_\H r$ as $r\to\infty$. Using this information, one easily infers from \eqref{eq:gn.global.asymp} that 
$$
G_N(x,A_{N-j,N-j+1})
\sim c_\H V(x)\frac{j}{N^{p}}\int_A h(z)dz 
$$
if $j\to\infty$ and $j=o(N)$.

We now formulate our result on the asymptotic behaviour of the harmonic measure.
\begin{theorem}
\label{thm:harm_measure} 
Assume that the random walk $S(n)$ and the cone $K$ satisfy the conditions of Theorem~\ref{thm:gN.global}.
If, additionally, $\E|X|^{p+d}$ is finite then, uniformly in
$y\in\Gamma_{N,\infty}$, 
\begin{align}
\label{eq:hm.local}
\nonumber
&\Pr_x(S(\sigma_N)=z)\\
&=\frac{V(x)}{N^{p+d-1}}h(y_N/N)
\E[v_\H(\text{dist}(y-X,\Gamma_N));y-X\in K_{0,N}]
+o\left(\frac{V(x)}{N^{p+d-1}}\right),
\end{align}
where $y_N\in\Gamma_N$ is such that $y$ and $y_N$ belong to the same ray 
in the cone $K$.
\end{theorem}

We will now move to the rate of convergence for the three gamblers problems. 
First we discuss a  Brownian motion analogue.  
Let 
\(B(t)=(B_{1}(t), B^{2}(t)\)  
be a standard Brownian motion. 
Let \(\mathcal T_N\) be an equilateral triangle with vertices 
\((0,0), (\sqrt 2 N,0), \left(\frac{1}{\sqrt 2}N,\sqrt{\frac{ 3}{2}} N\right)\). 
This model is the limiting case corresponding  to the linear transformation $TZ(n)$ 
of the original game. 
Then, starting from a position $(x_1,x_2)$ we run 
this two-dimensional Brownian motion until it hits 
one of the edges. When this happens one of the players 
gets eliminated and we run a one-dimensional 
Brownian motion on this edge until it hits one of the vertices. 
We will discuss the probability 
$P_{x_1,x_2}^{bm,(321)}(N)$ 
of the event that the 
third player gets eliminated first, then the second. 

It can be found as follows. 
Let \(\mathcal T_N^{(1)},\mathcal T_N^{(2)}\) and \(\mathcal T_N^{(3)}\) be the edges 
of this triangle, between correspondingly vertices 
\((\sqrt 2 N,0)\), \((0,0)\) and 
\(\left(\frac{1}{\sqrt 2}N, \sqrt{\frac{ 3}{2}} N\right)\). 
Consider the stopping times
\begin{align*}
\tau_N^{bm}&:=
\inf\{t>0:B(t)\notin \mathcal T_N\}\\
\sigma_N^{bm}&:=\inf\{t>0:B(t)\in \mathcal T^{(3)}_N\}. 
\end{align*}
Then, 
\[
    P_{x_1,x_2}^{bm,(321)}(N) = 
    \E_{x_1,x_2}\left[\frac{\sqrt{\frac{ 3}{2}} N-B_{2}(\tau_{N}^{bm})}{\sqrt{\frac{ 3}{2}} N};\tau_N^{bm}=\sigma_{N}^{bm}\right], 
    \text{ as } N\to \infty. 
\]
A solution to this problem 
was obtained earlier in~\cite{Hajek87} via 
a conformal mapping of the above triangle 
 to the unit disk. 
 Here we present a solution via a conformal mapping 
 of the triangle $\mathcal T_N$ to the upper half plane. 
 The form of the explicit solution, see~\eqref{eq.explicit} below, 
 allows us to find the  asymptotics  
 and the corresponding harmonic function. 
\begin{proposition}\label{prop.bm}
    For fixed $x_1,x_2$ 
    the following asymptotics hold, as $N\to\infty$, 
    \begin{align*} 
        &P_{x_1,x_2}^{bm,(321)}(N) \sim 
        \frac{\Gamma(1/3)^9}{96\sqrt 6 \pi^4}
        \frac{u(x)}{N^3},\\
        &\Pr_{x_1,x_2}(\text{third player gets eliminated first})    
        \sim 
        \frac{\Gamma(1/3)^9}{48\sqrt 6 \pi^4}
        \frac{u(x)}{N^3}.
    \end{align*}
where  $u(x)=3x_1^2x_2-x_2^3$  is the positive 
harmonic function in the wedge $K$ with the opening angle $\pi/3$.
\end{proposition}
Let $P_{x_1,x_2}^{(321)}(N)$ be the corresponding 
probability for the random walks in $\mathcal T_N$,  
that is that  
the third player gets eliminated first and the second player gets eliminated second.

\begin{corollary}\label{cor.rw} 
    Consider the random walk $Z(n)$. 
    For fixed $y_1, y_2$ 
    the following asymptotics hold, as $N\to\infty$, 
    \begin{align*} 
        &P_{y_1,y_2}^{(321)}(N) \sim 
        \frac{\Gamma(1/3)^9}{32 \pi^4}
        \frac{y_1 y_2(y_1+y_2)}{N^3},\\
        &\Pr_{y_1,y_2}(\text{third player gets eliminated first})    
        \sim 
        \frac{\Gamma(1/3)^9}{16 \pi^4}
        \frac{y_1 y_2(y_1+y_2)}{N^3}.
    \end{align*}
\end{corollary}    
\begin{proof} 
By arguments leading to~\eqref{eq:asymp.3.rw} the asymptotics 
is the same as the asymptotics for the Brownian motion given in Proposition~\ref{prop.bm}. 
We just need to 
recalculate the harmonic function in the original  coordinates. 
Recall that $x_1 = \frac{1}{\sqrt 2} y+\sqrt 2 y_2, 
x_2=\sqrt{3/2}y_2$.
This will result in 
\[
u(x_1,x_2)= 3\sqrt6
\frac{y_1y_2(y_1+y_2)}{N^3}.
\]
\end{proof}
We can now compare the theoretical result with 
computations in table 4 in~\cite{DE20}. 
The results from exact computation for $y_1=y_2=1$ 
and $N=50,100,150,200,250,300$ (rounded to $15$ significant figures) 
rapidly converge to a limit with $N^3 P_{1,1}^{(321)}(N) =4.55979450208$.  
The asymptotics in Corollary~\ref{cor.rw} is as follows,
\[
    P_{1,1}^{(321)}(N)  \sim
 \frac{\Gamma(1/3)^9}{16 \pi^4}\frac{1}{N^3}. 
\]
Since 
\[
    \frac{\Gamma(1/3)^9}{16 \pi^4}\approx 4.5597944999598458,    
\]
the answer is in excellent agreement with these exact computations.

Now we will move to the rate of convergence. 
It was conjectured in~\cite[Conjecture 4.2]{DE20} that 
for starting points $x_1,x_2$ 
sufficiently away from the edges 
\[
\left|
P_{x_1,x_2}^{(321)}(N)-P_{x_1,x_2}^{bm,(321)}(N) 
\right|    =O\left(\frac{1}{N^{4}}\right). 
\]
Using the estimates for the Green function 
we have established  the following rate of convergence. 
\begin{proposition}\label{prop.rate}
    There exists a constant $C$ such that 
    \[
        \left|
        P_{x_1,x_2}^{(321)}(N)-P_{x_1,x_2}^{bm,(321)}(N) 
        \right| \le \frac{C}{N^3}. 
        \]
\end{proposition}    
This rate convergence is quite fast, but not the same 
as conjectured in~\cite{DE20}. 
It is not clear to us which rate of convergence 
is the right one. 
However it is likely that  some symmetries of the problem 
can be used to obtain the conjectured rate of convergence.

\begin{remark}
    The problem with three gamblers can also be reformulated 
    as the problem of three non-colliding(ordered~\cite{DW10}) random walks. 
    For that consider for \(i=1,2,3\), the random walk 
\[
S^{(i)}_n=S^{(i)}_0+X_1^{(i)}+\cdots+ X_n^{(i)},\quad  n=1,2,\ldots 
\]
where \(\{X_j^{(i)}, j=1,2,\ldots , i=1,2,3\}\) are i.i.d. random variables distributed as \(X\), 
\[
\Pr(X=1) =\Pr(X=0)=\frac{1}{2}. 
\] 
Let $W=\{x=(x_1,x_2, x_3), x_1<x_2<x_3  \}$ be the Weyl chamber 
Consider the stopping time 
\[
    \tau:=\min \{n\ge1: (S^{(1)}_n,S^{(2)}_n, S^{(3)}_n)\notin W\}.         
\]
It is known that for these random walks 
the positive harmonic function on $W$ is given 
by the Vandermonde determinant $(x_3-x_2)(x_2-x_1)(x_3-x_1)$.  
Now let 
\[
\tau_N:=\min \{n\ge 1: S^{(3)}_n - S^{(1)}_n =N\}.   
\]
Then, the walk \((A_n,B_n)\) defined as 
\[
A_n = S^{(3)}_n - S^{(2)}_n, 
B_n =S^{(2)}_n - S^{(1)}_n,
\]
for \(n=0,1,2,3,\ldots \) will have the following transition probabilities 
\(p(x,y) =\Pr(A_n-A_{n-1}=x,B_n-B_{n-1}=y)\)
\begin{align*}
 p(0,1) =  p(0,-1)=p(1,0)=p(-1,0)
 =p(1,-1) =p(-1,1) = 
 \frac{1}{8},
 p(0,0) =\frac{1}{4}. 
\end{align*}
Thus, for \(n<\tau_N\wedge \tau\) 
the Markov chain \((A_n, B_n)\) will follow  the same path as the Markov chain 
$(Z(n))$. 
The only difference is that it will move slower, as it stays at the same place  with probability.
\(\frac{1}{4}\). 
This means that the distribution of the exit point is the same for 
\((A_n, B_n)\) and thus one can study the harmonic measure for non-colliding random walks 
to obtain the harmonic measure for $(Z(n))$.  
For non-colliding simple random walks there is an exact distribution for the exit time.   
Namely one can make  use of a discrete version of the Karlin-McGregor formula 
from~\cite{HW96}, see the last paragraph there. 
However, this exact formula involves an infinite number of reflections and seems 
to be difficult for the asymptotic analysis or numerical computations.

\end{remark}

In the rest of the paper we give proofs of the above statements. 
Proof of each statement is given in a separate section.  

\textbf{ Acknowledgement.}  
We are grateful  to Persi Diaconis for drawing 
attention to this problem and a number 
of encouraging discussions. 

\section{Proof of Theorem~\ref{thm:gN.global}}
Fix $\varepsilon>0$ and 
 split the Green function $G_N$ into two parts: 
 \[
    G_{N}(x,y) := G_{N}^{(1)}(x,y)+G_{N}^{(2)}(x,y),     
 \]
 where 
\begin{align*}
    G_{N}^{(1)}(x,y)&:=\sum_{n\le\varepsilon N^2}\Pr_x(S(n)=y,\sigma_N>n)\\ 
    G_{N}^{(2)}(x,y)&:=\sum_{n>\varepsilon N^2}\Pr_x(S(n)=y,\sigma_N>n). 
\end{align*}
We will start with analysis of $G_{N}^{(1)}(x,y)$. This part corresponds to large deviations of the random walk.  
\begin{lemma}\label{lem:ld.A}
    Assume that $\E|X|^p<\infty$. 
    There exists a function $f:\mathbb R\to [0,1]$  satisfying  
    $\lim_{\varepsilon\to0}f(\varepsilon)=0$ such that 
    \[
        G_{N}^{(1)}(x,A_{N-j,N-j+1})\le C_A f(\varepsilon) \frac{V(x)}{N^p}j.
    \] 
\end{lemma}    
\begin{proof}
    Let 
    \[
        \theta:=\inf\{k\ge 1\colon S(k)\in A_{N-j,N-j+1}\}.     
    \] 
    Then, by the strong Markov property,
    \begin{multline*} 
        G_{N}^{(1)}(x,A_{N-j,N-j+1}) 
        =\E_x\left[
            \sum_{n=\theta}^{\varepsilon N^2}
            \ind{S(n)\in A_{N-j,N-j+1},\sigma_N>n};
            \theta\le \varepsilon N^2
        \right]\\ 
        \le 
        \E_x\left[
            \E_{S(\theta)}\left[\sum_{n=0}^{\infty}
            \ind{S(n)\in A_{N-j,N-j+1},\sigma_N>n}
            \right];
            \theta\le \varepsilon N^2,\tau>\theta
        \right]. 
    \end{multline*}
  For every $y\in \Gamma_{N-j,N-j+1}$, 
  \begin{multline*} 
    \E_y
    \left[\sum_{n=0}^{\infty}
            \ind{S(n)\in A_{N-j,N-j+1},\sigma_N>n}
    \right]\\ 
    \le 
    \E_{\text{dist}(y,\H)}\left[
        \sum_{n=0}^{\infty}
        \ind{S_\H(n)\in [j-1,j]};\tau_\H>n
    \right]\le Cj,
  \end{multline*}
  see e.g.~\cite[Chapter 19]{Spitzer64} for the latter inequality for the Green function 
  of one-dimensional random walk on a half-line. 
  
  Therefore, 
  \[
    G_N^{(1)}(x, A_{N-j,N-j+1})\le Cj \Pr_x(\theta\le \varepsilon N^2,\tau>\theta).
  \]
  Applying the Doob $h$-transform with the positive harmonic function $V$, 
  we have 
  \[
    \Pr_x(\theta\le \varepsilon N^2,\tau>\theta) 
    =V(x) \widehat \E_x\left[
        \frac{1}{V(S(\theta))};\theta\le \varepsilon N^2
    \right].
  \]
  Since $A\subset\text{int}(\Gamma)$ the distance 
  $\text{dist}(A_{N-j,N-j+1})\ge c_A N$ for all $j\le N/2$. 
  Therefore, 
  $\inf_{y\in A_{N-j,N-j+1}} V(y)\ge \widehat c_A N^p$. 
  This implies that 
    \[
        \Pr_x(\theta\le \varepsilon N^2,\tau>\theta)  \le C_A 
        \frac{V(x)}{N^p} 
        \widehat \Pr(\theta\le \varepsilon N^2). 
  \]
  By the functional limit theorem for the 
  random walk conditioned to stay in a cone, see Theorem 3 in \cite{DuW20}, 
  $\widehat \Pr_x(\theta\le \varepsilon N^2)\le f(\varepsilon)$ for some 
  $f$ with required properties. 
\end{proof}    

\begin{lemma}\label{lem:ld}
    Assume that $\E|X|^{p+d}<\infty$. 
    There exists a function 
    $f:\mathbb R\to [0,1]$  satisfying  
    $\lim_{\varepsilon\to0}f(\varepsilon)=0$ such that 
    \[
        G_{N}^{(1)}(x,y)
        \le 
        C f(\varepsilon) 
        \frac{V(x)}{N^{p+d-1}}j 
    \] 
    uniformly in $y\in A_{N-j,N-j+1}$ and uniformly in $A\subset\Gamma$. 
\end{lemma}    

\begin{proof}
    Fix some $\delta>0$ and let 
    \[
      \theta_y:=\inf\{k\ge 1\colon |S(k)-y|\le \delta N\}.   
    \]
    Repeating the same arguments as in Lemma~\ref{lem:ld.A}, we obtain 
    \[
        G_{N}^{(1)}(x,y)
        \le 
        \E_x\left[
            \E_{S(\theta_y)} 
            \left[
            \sum_{n=0}^\infty 
            \ind{S(n)=y},\tau_\H>n
            \right];
            \tau>\theta_y, \theta_y\le \varepsilon N^2 
        \right]  
    \]
    Since $\E|X|^{d+1}<\infty$ we can apply 
    Theorem~1.2 from~\cite{DRTW20}, which implies that 
    \[
        \E_{S(\theta_y)} 
        \left[
        \sum_{n=0}^\infty 
        \ind{S(n)=y},\tau_\H>n
        \right]
        \le C \frac{1+\text{dist}(S(\theta_y), N\Gamma)}{1+|S(\theta_y)-y|^d}j.    
    \]
    Therefore, 
    \begin{align*}
        G^{(1)}_N(x,y) 
        &\le 
        \frac{CNj}{N^d}
        \Pr_x\left(
            |y-S(\theta_y)|>\frac{\delta}{2}N, 
            \tau>\theta_y, \theta_y\le \varepsilon N^2
        \right)\\
        &\hspace{1cm}+ CNj
        \Pr_x\left(
            |y-S(\theta_y)|\le \frac{\delta}{2}N, 
            \tau>\theta_y, \theta_y\le \varepsilon N^2
        \right)\\ 
        &\le 
        C 
        \frac{Cj}{N^{d-1}} 
        \Pr_x\left(|y-S(\theta_y)|>\frac{\delta}{2}N,
             \tau>\theta_y, \theta_y\le \varepsilon N^2
        \right)\\
        &\hspace{1cm}+ 
        CNj \Pr(|X(\theta_y)|>\frac{\delta}{2}N)\\
        &\le 
        \frac{Cj}{N^{d-1}} 
        \frac{V(x)}{\inf_{z\in K\colon \frac{\delta}{2}N\le|z-y|\le \delta N} V(z)}
        \widehat \Pr_x\left(
             \theta_y\le \varepsilon N^2
        \right)\\
        &\hspace{1cm}+ 
        CNj \Pr(|X(\theta_y)|>\frac{\delta}{2}N).
    \end{align*}
    Making use of the lower bound for the harmonic function 
    $$
    \inf_{z\in K\colon \frac{\delta}{2}N\le|z-y|\le \delta N} V(z)
    \ge c_\delta N^{p}
    $$ 
    and the moment assumption $\E|X|^{p+d}<\infty$ we arrive at the conclusion. 
\end{proof}    

In order to determine the behaviour of $G_N^{(2)}$ we study each probability separately.

Let $f_r$ be the density of the measure 
$\Pr\left(M_K(1)\in\cdot, M_K(t)\in K_{0,1/r}\text{ for all }t\le 1\right)$. 
This meander $M_K$ is defined and studied in~\cite{DuW20}, 
where the functional Central Limit Theorem that we are using in the next lemma is proved. 

\begin{lemma}\label{lem:asymp.pn}
    Let $n$ be such that $\frac{n}{N^2}\to r^2$ as $N\to \infty$. 
    Then, 
    \begin{equation}
        \label{eq:asymp.pn}
        \sup_{y\in K_{0,N}}\left|n^{p/2+d/2}\Pr_x\left(S(n)=y,\sigma_{N}>n\right)
        -\varkappa V(x) f_r\left(\frac{y}{\sqrt{n}}\right)\right|\to0. 
    \end{equation}    
\end{lemma}

\begin{proof}
    Applying  Theorem 1 in \cite{DuW20}, one concludes easily that 
    \begin{multline*}
        \Pr_x\left(
            S(n)\in \sqrt n D, S(k)\in K_{0,N} 
            \text{ for all } k\le n   
            \mid \tau>n
        \right)\\ 
        \to 
        \Pr(M_K(1)\in D, M_K(t)\in K_{0,1/r}
        \text{ for all } t\le 1 ).
    \end{multline*}
Arguing exactly in the same way as in~\cite[Theorem 5]{DW15} 
we obtained the desired local limit theorem from the above 
global limit theorem. 
\end{proof}

\begin{lemma}\label{lem:asympt.pn}
There exists an independent of the random walk function $g(z,r)$ such that
    \begin{align*}
        &\frac{n^{p/2+d/2+1/2}}{V(x)v_{\H}(\text{dist}(y,\Gamma_N))}\Pr_x\left(S(n)=y,\sigma_{N}>n\right) \\
        &\hspace{2cm}=\varkappa\sqrt{\frac{2}{\pi}}2^{p/2+d/2} 
        g\left(\frac{y}{\sqrt{n}},\frac{N}{\sqrt{n/2}}\right)+o(1).
    \end{align*} 
    uniformly in $\frac{n}{N^2}\in[\varepsilon,\varepsilon^{-1}]$
    and in $\text{dist}(y,N\Gamma)=o(N)$.
    
Furthermore, uniformly in $y\in K_{0,N}$,
$$
\Pr_x\left(S(n)=y,\sigma_{N}>n\right)\le
C \frac{V(x)v_{\H}(\text{dist}(y,\Gamma_N))}{n^{p/2+d/2+1/2}}.
$$
\end{lemma}    
\begin{proof} 
Set $m=[n/2]$. We first make use of the time inversion as follows,  
\begin{multline}
    \label{time-inv}
    \Pr_x\left(S(n)=y,\sigma_{N}>n\right)\\
    =\sum_{z}\Pr_x\left(S(m)=z,\sigma_{N}>m\right)
    \Pr_z\left(S(n-m)=y,\sigma_{N}>n-m\right)\\
    =\sum_{z}\Pr_x\left(S(m)=z,\sigma_{N}>m\right)
    \Pr_y\left(\widehat S(n-m)=z,\widehat{\sigma}_{N}>n-m\right),
    \end{multline}
    where $\widehat S(k)=-X(1)-X(2)-\ldots-X(k)$, $k\ge1$ 
    is a time-reversed random walk and $\widehat \sigma_N$ is the corresponding stopping time for $(\widehat S(n))_{n\ge 0}$, that is
    $$
    \widehat \sigma_N=\inf\{n\ge1: \widehat{S}(n)\notin K_{0,N}\}.
    $$
    Define also the exit time of $\widehat{S}(n)$ from the half space with the boundary $\H_N$:
    $$
    \widehat{\tau}_{\H}:=\inf\{n\ge 1: (\widehat{S}(n),x_0)\ge N\}.
    $$
    It is clear that 
    $$
    \Pr_y(\widehat{\tau}_{\H}>n)
    =\Pr_{\text{dist}(y,\Gamma_N)}(\tau_{\H}>n).
    $$
    Using the known results for one-dimensional walks, we have
    \begin{equation}
    \label{eq:tau1}
    \Pr_y(\sigma_{N}>n-m)\le \Pr_y(\widehat \tau_H>n)
    \le C\frac{1+\text{dist}(y,\Gamma_N)}{\sqrt n} 
    \quad\text{uniformly in }y
    \end{equation}
    and, uniformly in $y$ with $\text{dist}(y,\Gamma_N)=o(\sqrt{n})$,
    \begin{equation}
    \label{eq:tau2}
    \Pr_y(\widehat \tau_H>n)\sim \sqrt{\frac{2}{\pi}} \frac{v_\H(\text{dist}(y,N\Gamma))}{n^{1/2}}.
    \end{equation}

    Combining Lemma~\ref{lem:asymp.pn} with \eqref{eq:tau1}, we obtain
    \begin{multline*}
        \frac{n^{p/2+d/2}}{V(x)}\Pr_x\left(S(n)=y,\sigma_{N}>n\right) \\ 
        =
        \varkappa 2^{p/2+d/2}\E_y
        \left[
            f_{n/\sqrt 2 N} \left(
                \frac{\widehat S(n-m)}{\sqrt m}\right); \widehat \sigma_N>n-m
            \right] 
            +o\left(
                \text{dist}(y,\Gamma_N)
            \right).
    \end{multline*}
  
    By the functional CLT for random walks conditioned to stay in a half-space, 
    \[
        \E_y\left[ f_{n/\sqrt 2N} \left(
                \frac{\widehat S(n-m)}{\sqrt m}\right)\ind{\widehat \sigma_N>n-m}\Big|\widehat{\tau}_{H}>n
            \right] 
            -g\left(\frac{y}{\sqrt{n}},\frac{n}{\sqrt 2N}\right)=o(1)
    \]        
    where 
    \[
      g(z,r) = \E_{z_r}\left[
        f_{r/\sqrt 2} (M_{\H} (1)); M_\H(t)\in K_{0,1/r}
      \right],  
    \]
    where $z_r\in \Gamma_{1/r}$ is such that $z$ and $z_r$ belong to the same ray. Clearly, this function is continuous in both coordinates.
    
    Combining this with \eqref{eq:tau2}, we have 
    \begin{align*}
        &\frac{n^{p/2+d/2+1/2}}{V(x)v_{\H}(\text{dist}(y,\Gamma_N))}\Pr_x\left(S(n)=y,\sigma_{N}>n\right) \\
        &\hspace{2cm}=\varkappa\sqrt{\frac{2}{\pi}}2^{p/2+d/2} 
        g\left(\frac{y}{\sqrt{n}},\frac{N}{\sqrt{n/2}}\right)+o(1).
    \end{align*}
To prove the upper bound we notice that, by Lemma 27 from \cite{DW15}, 
$$
\Pr_x\left(S(m)=z,\sigma_{N}>m\right)
\le \Pr_x\left(S(m)=z,\tau>m\right)\le C\frac{V(x)}{m^{p/2+d/2}}.
$$
Plugging this into \eqref{time-inv}, we conclude that 
$$
\Pr_x\left(S(n)=y,\sigma_{N}>n\right)
\le C\frac{V(x)}{m^{p/2+d/2}}\Pr_y(\sigma_{N}>n-m).
$$
Applying now \eqref{eq:tau1} completes the proof.
\end{proof}

\begin{proof}[Completion of the proof of Theorem~\ref{thm:gN.global}] 
    Making use of Lemma~\ref{lem:ld.A} and Lemma~\ref{lem:ld} 
    with $\varepsilon = \infty $ we obtain the upper bound~\eqref{eq:gN.upper.bound} and its local version. 

We will find now  asymptotics for $G^{(2)}_N(x,y)$. 
According to the first part of Lemma~\ref{lem:asympt.pn},
\begin{align}
\label{eq:t1.1}
\nonumber
&\frac{1}{V(x)v_\H(\text{dist}(y,N\Gamma))}
\sum_{n=\varepsilon N^2}^{N^2/\varepsilon} \Pr_x(S(n)=y,\sigma_N>n)\\
&\hspace{2cm}
=c_0\sum_{n=\varepsilon N^2}^{N^2/\varepsilon}n^{-p/2-d/2-1/2}
g\left(\frac{y}{\sqrt{n}},\frac{N}{\sqrt{n/2}}\right)
+o\left(\frac{1}{N^{p+d-1}}\right).
\end{align} 
Note next that, as $N\to\infty$, 
\begin{align}
\label{eq:t1.2} 
\nonumber
&\sum_{n=\varepsilon N^2}^{N^2/\varepsilon}n^{-p/2-d/2-1/2}
g\left(\frac{y}{\sqrt{n}},\frac{N}{\sqrt{n/2}}\right)\\
\nonumber
&\hspace{1cm}
=\frac{1}{N^{p+d-1}}\sum_{n=\varepsilon N^2}^{N^2/\varepsilon}
\left(\frac{N}{\sqrt{n}}\right)^{p+d+1}
g\left(\frac{y}{N}\frac{N}{\sqrt{n}},\frac{N}{\sqrt{n/2}}\right)
\frac{1}{N^2}\\
\nonumber
&\hspace{1cm}
=\frac{2}{N^{p+d-1}}\sum_{n=\varepsilon N^2}^{N^2/\varepsilon}
\left(\frac{N}{\sqrt{n}}\right)^{p+d}
g\left(\frac{y}{N}\frac{N}{\sqrt{n}},\frac{N}{\sqrt{n/2}}\right)
\frac{1}{N^2}\frac{N}{2\sqrt{n}}\\
&\hspace{1cm}
=\frac{2+o(1)}{N^{p+d-1}}\int_{\varepsilon^{1/2}}^{\varepsilon^{-1/2}}
r^{-p-d}g\left(\frac{y}{N}r^{-1},\sqrt{2}r^{-1}\right)dr
\end{align}
Finally, using the uniform upper bound from Lemma~\ref{lem:asymp.pn}, we have 
\begin{align*}
\sum_{n=N^2/\varepsilon}^\infty\Pr_x(S(n)=y,\sigma_N>n)
&\le CV(x)v_\H(\text{dist}(y,N\Gamma))
\sum_{n=N^2/\varepsilon}^\infty n^{-p/2-d/2-1/2}\\
&\le C\varepsilon^{p/2+d/2-1/2}\frac{V(x)v_\H(\text{dist}(y,N\Gamma))}{N^{p+d-1}}.
\end{align*}
Combining this with \eqref{eq:t1.1} and \eqref{eq:t1.2}, we obtain
\begin{align*}
&\limsup_{N\to\infty}\left|
\frac{N^{p+d-1}}{V(x)v_\H(\text{dist}(y,N\Gamma))}G_N^{(2)}(x,y)
\right.\\
&\hspace{2cm}-\left.2c_0\int_{\varepsilon^{1/2}}^\infty 
r^{-p-d}g\left(\frac{y}{N}r^{-1},\sqrt{2}r^{-1}\right)dr\right|
\le C \varepsilon^{p/2+d/2-1/2}.
\end{align*}
Combining this estimate and with the bound from Lemma~\ref{lem:ld}
and letting $\varepsilon\to0$, we infer that \eqref{eq:gn.local.asymp}
holds with 
\begin{equation}
\label{eq:h-def}
h(z)=2c_0\int_0^\infty r^{-p-d}g(zr^{-1},\sqrt{2}r^{-1})dr.
\end{equation}
\end{proof}    
\section{Proof of Theorem~\ref{thm:harm_measure}.}
According to the total probability formula,
\begin{align}
\label{eq:hm.11}
\nonumber
\Pr_x(S(\sigma_N)=y)&=\sum_{z\in K_{0,N}}
\sum_{n=0}^\infty\Pr_x(S(n)=z;\sigma_N>n)\Pr(z+X=y)\\
&=\sum_{z\in K_{0,N}}G_N(x,z)\Pr(y-X=z).
\end{align}
We first notice that  
\begin{align*}
\sum_{z\in K_{0, N/2}}G_N(x,z)\Pr(y-X=z)
\le \Pr(|X|>N/2)\max_{z\in K}G_N(x,z).
\end{align*}
Applying Lemma 27 from \cite{DW15}, we have
\begin{align*}
G_N(x,z)
&\le 1+\sum_{n=1}^\infty\Pr_x(S(n)=z;\tau>n)\\
&\le 1+C(x)\sum_{n=1}^\infty n^{-p/2-d/2}\le C'(x)
\quad\text{uniformly in }z\in K.
\end{align*}
Due to the assumption $\E|X|^{p+d}<\infty$, $\Pr(|X|>N/2)=o(N^{-p-d})$.
Therefore,
\begin{equation}
\label{eq:hm.12}
\sum_{z\in K_{0, N/2}}G_N(x,z)\Pr(y-X=z)=o(N^{-p-d}).
\end{equation}

Fix now a sequence $R_N$ such that $R_N\to\infty$ and $R_N=o(N)$.
According to the upper bound \eqref{eq:gn.local.bound},
\begin{align*}
 &\sum_{z\in \Gamma_{N/2,N-R_N}}G_N(x,z)\Pr(y-X=z)\\
 &\hspace{1cm}\le C\frac{V(x)}{N^{p+d-1}}
 \sum_{z\in \Gamma_{N/2,N-R_N}}\text{ dist}(z,\Gamma_N)\Pr(y-X=z)
\end{align*}
Since the second moment of $|X|$ is finite,
\begin{equation}
\label{eq:hm12a}
\sum_{z\in \Gamma_{0,N-R_N}}\text{ dist}(z,\Gamma_N)\Pr(y-X=z)
=o(1).
\end{equation}
Consequently,
\begin{equation}
\label{eq:hm.13}
\sum_{z\in \Gamma_{N/2,N-R_N}}G_N(x,z)\Pr(y-X=z)
=o(N^{-p-d+1}).
\end{equation}
Since $R_N=o(N)$, for $z\in\Gamma_{N-R_N,N}$ we may apply \eqref{eq:gn.local.asymp}:
\begin{align*}
 &\sum_{z\in \Gamma_{N-R_N,N}}G_N(x,z)\Pr(y-X=z)\\
 &\hspace{1cm}=\frac{V(x)}{N^{p+d-1}}\sum_{z\in \Gamma_{N-R_N,N}}
 v_\H(\text{ dist}(z,\Gamma_N))h(z/N)\Pr(y-X=z)\\
 &\hspace{2cm}+o\left(\frac{V(x)}{N^{p+d-1}}
 \sum_{z\in \Gamma_{N-R_N,N}}
 v_\H(\text{ dist}(z,\Gamma_N))\Pr(y-X=z)\right).
\end{align*}
The finiteness of the second moment of $|X|$ implies that 
$$
\sum_{z\in \Gamma_{N-R_N,N}}
 v_\H(\text{ dist}(z,\Gamma_N))\Pr(y-X=z)=O(1)
$$
and that 
\begin{align*}
&\sum_{z\in \Gamma_{N-R_N,N}}
 v_\H(\text{ dist}(z,\Gamma_N))h(z/N)\Pr(y-X=z)\\
&\hspace{1cm}=h(y_N/N)\sum_{z\in \Gamma_{N-R_N,N}}
 v_\H(\text{ dist}(z,\Gamma_N))\Pr(y-X=z)+o(1).
\end{align*}
(Recall that $y_N$ is the point on $\Gamma_N$ which lies on the same ray as $y$.)

Therefore,
\begin{align*}
 &\sum_{z\in \Gamma_{N-R_N,N}}G_N(x,z)\Pr(y-X=z)\\
 &\hspace{1cm}=\frac{V(x)}{N^{p+d-1}}h(y_N/N)
 \sum_{z\in \Gamma_{N-R_N,N}}
 v_\H(\text{ dist}(z,\Gamma_N))\Pr(y-X=z)
 +o\left(\frac{V(x)}{N^{p+d-1}}\right).
\end{align*}
Combining this with \eqref{eq:hm12a}, we finally get 
\begin{align}
\label{eq:hm.14}
\nonumber
&\sum_{z\in \Gamma_{N-R_N,N}}G_N(x,z)\Pr(y-X=z)\\
&=\frac{V(x)}{N^{p+d-1}}h(y_N/N)
\E[v_\H(\text{dist}(y-X,\Gamma_N));y-X\in K_{0,N}]
+o\left(\frac{V(x)}{N^{p+d-1}}\right).
\end{align}
Pugging \eqref{eq:hm.12}, \eqref{eq:hm.13} and \eqref{eq:hm.14} into 
\eqref{eq:hm.11}, we get \eqref{eq:hm.local}.

 \section{Three gambler's problem: continuous case}
In this Section we give a proof of Proposition~\ref{prop.bm}.   
Let $x=(x_1,x_2)$. 
To find the probability of interest
 we need to solve the following Dirichlet 
problem 
\begin{equation}
    \label{eq:laplace}
\begin{cases}    
    \Delta v_N(x) = 0,& x\in \mathcal T_N\\
    v_N(x) = 0, & x\in \mathcal T_N^{(2)}\cup \mathcal T_N^{(1)},\\
    v_N(x) = N\phi(x/N), & x\in \mathcal T_N^{(3)}.
\end{cases}
\end{equation}
Then the solution to this problem with the boundary condition 
\(\phi(x)=
1-\sqrt{\frac{2}{3}}\frac{x_2}{N}
\)
will give $P_{x_1,x_2}^{bm,(321)}(N)=\frac{v_N(x)}{N}$. 
The solution of the problem with the boundary condition  
$\phi(x)=1$ will results in 
\[
\Pr_{x_1,x_2}(\text{third player gets eliminated first})    
=\frac{v_N(x)}{N}.
\]

Next note also that by scaling 
\begin{equation}\label{eq:scaling}
v_N(x) = N v_1(x/N), \quad x\in \mathcal T_N\cup 
\partial \mathcal T_N. 
\end{equation}

Problem~\eqref{eq:laplace}   can be solved using the conformal mappings. 
In view of the scaling it is sufficient to consider the conformal mapping of the 
triangle \(\mathcal T_{1/\sqrt 2}\) to the upper half plane.  
This mapping is given by the formula
\[
    w(z) = \frac{1}{2}+\frac{27}{2B\left(\frac 13,\frac 13\right)^3}
    \wp'\left(\overline z e^{-\pi i/3};0,-\frac{1}{27^2}B\left(\frac 13,\frac 13\right)^6\right)
\]
and \(\wp'\) is the first derivative of Weierstrass's elliptic function. 
This mapping transforms the edge from $0$ to $1$ into the half line from $1$ to $+\infty$, 
the edge from $0$ to $e^{\pi i/3}$ into the half line from $-\infty$ to $0$ and 
the edge from $1$ to $e^{\pi i/3}$ into the half line from $1$ to $0$. 
The inverse mapping from the upper half plane to the triangle is given by 
\[
z(w) = \overline{\frac{1}{B\left(\frac 13,\frac 13\right)}\int_0^w \frac{dt}{t^{2/3}(1-t)^{2/3}} } e^{-\pi i/3}+  e^{\pi i/3}. 
\]
On the half-plane 
the solution to the Dirichlet problem 
is given by  the Poisson kernel for the half-plane, 
see~\cite[Theorem 1.7.2]{AG01}.  
As a result the solution to~\eqref{eq:laplace}  can be written down 
as follows  
\begin{align}
    \nonumber 
v_{1/{\sqrt 2}}(z) &= \frac{1}{\pi\sqrt 2} \int_{-\infty}^\infty 
\frac{\textrm{Im} w(z) }{|t-w(z)|^2} \phi(z(t))dt \\
&=   \frac{1}{\pi \sqrt 2} \int_{0}^1
\frac{\textrm{Im} w(z) }{|t-w(z)|^2} \phi(z(t))dt.\label{eq.explicit}
\end{align}
This solution is harmonic in the triangle 
and continuous at its boundary at points,  
where the function $\phi$ is continuous.

Plug in now  
the  initial condition 
$z=\frac{1}{N}(x_1+ i x_2)$.  
Since $\wp'$ has a pole of order $3$ at $0$, we have 
\[
    w(z) \sim -\frac{27}{B\left(\frac 13,\frac 13\right)^3} \frac{N^3}{(x_1+ix_2)^3}.   
\]
Then, uniformly in $t\in (0,1)$, 
\begin{align*}
    \frac{\textrm{Im} w(z) }{|t-w(z)|^2} &\sim 
    \frac{\textrm{Im} w(z) }{|w(z)|^2}
    \sim 
    -\frac{B\left(\frac 13,\frac 13\right)^3} {27 N^3}
    |x_1+ix_2|^6 \textrm{Im}\left(\frac{1}{ (x_1+ix_2)^3 }\right)
    \\
    &
    =
    -\frac{B\left(\frac 13,\frac 13\right)^3} {27 N^3}
    \textrm{Im} (x_1-ix_2)^3
    =\frac{B\left(\frac 13,\frac 13\right)^3} {27 N^3}
    (3x_1^2x_2-x_2^3). 
\end{align*}
Hence, 
\[
    v_{1/{\sqrt 2}}(z)
    \sim     
    \frac{B\left(\frac 13,\frac 13\right)^3} {27 N^3}
    (3x_1^2x_2-x_2^3)\frac{1}{\pi \sqrt 2}
    \int_0^1 \phi(z(t))dt.
\]
Plugging in the initial condition $\phi(x)=1$ we obtain 
\[
    v_{1/{\sqrt 2}}(z)
    \sim     
    \frac{B\left(\frac 13,\frac 13\right)^3} {27\sqrt 2 \pi N^3}
    (3x_1^2x_2-x_2^3)
\]
Using the scaling we obtain, 
\begin{multline*} 
        \Pr_{x_1,x_2}(\text{third player gets eliminated first})    
        =v_1(x/N) \\= \sqrt 2 v_{1/\sqrt{2}}(x/(\sqrt 2 N)) 
        \sim 
        \frac{B\left(\frac 13,\frac 13\right)^3} {54\sqrt 2\pi N^3}
    (3x_1^2x_2-x_2^3)=
    \frac{\Gamma(1/3)^9}{48\sqrt 6 \pi^4}\frac{u(x)}{N^3}.
\end{multline*}
Next we plug in  boundary condition \(\phi(x)=
1-\frac{2x_2}{\sqrt{3}}\).
Then,            
\begin{align*} 
    \int_0^1 \phi(z(t))dt 
    &=
    \int_0^1 \left(1-\frac{2}{\sqrt{3}}\text{Im} z(t)\right)dt \\ 
    &=
    \frac{1}{B(1/3,1/3)}
    \int_0^1\int_0^t \frac{du}{u^{2/3}(1-u)^{2/3}}dt.
\end{align*}
Integrating by parts we obtain 
\begin{align*} 
    \int_0^1 \phi(z(t))dt = 
    \frac{1}{B(1/3,1/3)}
    \int_0^1 (1-u)^{1/3} u^{-2/3}du 
    = 
    \frac{B(4/3,1/3)}{B(1/3,1/3)}
\end{align*}
Then, 
\begin{align*}
    v_{1/{\sqrt 2}}(z)\sim 
    \frac{B\left(\frac 13,\frac 13\right)^2B\left(\frac 43,\frac 13\right)} {27\sqrt{2}\pi N^3}
     (3x_1^2x_2-x_2^3)\\\sim         
     \frac{B\left(\frac 13,\frac 13\right)^3} {54\sqrt 2\pi  N^3}
     (3x_1^2x_2-x_2^3)
     &=
     \frac{\Gamma(1/3)^9}{48\sqrt 3 \pi^4 \sqrt 2} 
     \frac{3x_1^2x_2-x_2^3}{N^3}
\end{align*}
Using the scaling we obtain 
\[
    v_{1}(z) \sim \frac{\Gamma(1/3)^9}{96\sqrt 6 \pi^4}\frac{u(x)}{N^3},
\]
which implies the statement.

\section{Rate of convergence: proof of Proposition~\ref{prop.rate}}

\subsection{Extension of the harmonic function}
We will first extend the harmonic function $v_N(x)$ to obtain better estimates for its derivatives.  
Let $\mathcal H_N$ is the hexagon obtained by rotation 
of $\mathcal T_N$ about origin $5$ times by $\frac{\pi}{3}$  each time. 
Let $\widetilde T_N$ be the triangle obtained by the union of the reflection 
of $\mathcal T_N$ with respect the edge $\mathcal T_N^{(3)}$ and 
the edge $\mathcal T_N^{(3)}$. 
\begin{figure}
\begin{tikzpicture}[scale =1.5]
    \draw[thick]    (0,0) -- node[above]{$\mathcal T_N$} ++(0:1) -- ++(120:1) -- ++(240:1) --cycle;
    \draw (1,0) -- ++(60:1) -- node[below]{$\widetilde{\mathcal T_N}$} ++(180:1) --  ++(180:1) -- ++ (240:1)  
     -- ++(300:1) -- ++(0:1) --cycle;
\end{tikzpicture} 
\caption{Region $\mathcal H_N\cup \widetilde{\mathcal T_N}$ to which $v_N$} 
is extended 
\label{fig.hn}
\end{figure}
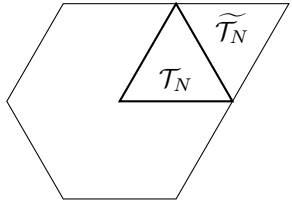
The resulting region $\mathcal H_N \cup\widetilde {\mathcal T}_N$ 
can be seen at Figure~\ref{fig.hn}. 
\begin{lemma}\label{lem:extension}
    Function $v_N$ can be extended to 
    $\mathcal H_N \cup\widetilde {\mathcal T}_N$
    in such a way that it is harmonic on this region. 
\end{lemma}    
\begin{proof} 
Note that using the standard Schwartz reflection principle (see~\cite[Theorem 1.3.6]{AG01}) 
we can construct a harmonic extension of  
the function \(v_N(x)\) over \(\mathcal T_N^{(1)}\) and \(\mathcal T_N^{(2)}\). 
For that note that    $v_N(x)=0$ 
for $x\in \mathcal T_N^{(1)}\cup \mathcal T_N^{(2)}$ 
and  is continuous 
at the boundary of $\mathcal T_N$ 
except vertices of  \(\mathcal T_N^{(3)}\). 
Then Theorem~1.7.5  implies continuity 
of  $v_N(x)$ on the closure $\mathcal T_N$ except 
vertices of  \(\mathcal T_N^{(3)}\) and hence~\cite[Theorem 1.3.6]{AG01} is applicable. 

The construction of the reflection (and prove of harmonicity) is as follows. 
As  $v_N(x)=0$ for $x\in \mathcal T_N^{(1)}\cup \mathcal T_N^{(2)}$ and 
we can extend it to the reflection of the triangle over the line $\{x_2=0\}$ by the usual formula 
\[
v_N(x_1,x_2)=-v_N(x_1,-x_2), \quad x_2 \in -\mathcal T_N.    
\] 
The resulting function is indeed harmonic as for $x\in \mathcal T_N$ or 
$x\in -\mathcal T_N$
it is equal to the average over all sufficiently small 
balls about $x$. For $x$ at the boundary such 
that $x_2=0$ we note that the average over all sufficiently small balls is equal 
to $0$ by cancellations in our construction and hence  is equal 
to the value of the function. 
Since $v_N$ is equal to the average over all small balls 
for all $x$ in the region under consideration it is harmonic.

The reflection over other side of the triangle with $0$ boundary conditions is analogous. 
Applying the reflection principle several times we 
obtain that function \(v_N(x)\) has a  harmonic continuation 
to  the hexagon  \(\mathcal H_N\). 

It is also possible to extend the side with non-zero boundary conditions. 
Indeed, rotating the triangle to simplify the notation  
we can assume that non-zero boundary conditions are on the side 
connecting $(0,0)$ and $(\sqrt 2 N, 0))$ and are given by 
\[
v_N(x_1,0) = \frac{x_1}{\sqrt 2 N}.     
\]
Note that  function $\frac{x_1}{\sqrt 2 N}$ is harmonic over the whole plane. 
Now put 
\[
    \widetilde v_N(x_1,x_2) = v_N(x_1,x_2) -\frac{x_1}{\sqrt 2 N}. 
\]
Function $\widetilde v_N(x_1,x_2)$ is harmonic 
over $\mathcal T_N$ and is equal to $0$ at the boundary. 
Hence we can extend it to $-\mathcal T_N$ by the same formula  
\[
\widetilde v_N(x_1,x_2)=-\widetilde v_N(x_1,-x_2), \quad x_2 \in -\mathcal T_N.    
\]
Then, 
\[
    v_N(x_1,x_2)=\widetilde v_N(x_1,x_2)+\frac{x_1}{\sqrt 2 N}
\]
is an harmonic extension of the original function with the required boundary conditions. 
Thus, we have shown that the extension exist to the above hexagon and 
the adjacent equilateral triangle with the 
side given by $\mathcal T_N^{(3)}$.
\end{proof}

\subsection{Diffusion approximation}

For $y\in \mathcal T_N$ let $\delta(y)$ 
be the distance from  $y$ to the closest vertex of the triangle.  
Let  
$\mathcal C_n$ be the convex hull of 
\[ 
(\overline{\mathcal T}_N \cap T\mathbb Z^2)\setminus  \{\text{vertices of $\mathcal T_n$} \}. 
\] 
Using the above extension of the harmonic function we can bound the derivatives of the 
harmonic function $v_N$ in exactly the same way as in~\cite[Lemma 7]{DW15} 
to obtain.  

\begin{lemma}\label{lem:derivatives}
    For any $k$ there exists a constant $c_k$ such that 
    for $x\in \mathcal C_N$ 
    and $\alpha:|\alpha|\le k$, 
    \begin{equation}
        \left|\frac{\partial^\alpha v_N(x)}{\partial x^{\alpha}}\right| 
        \le c_k\frac{N}{\delta(x)^{|\alpha|}}
        \label{eq:derivative0}
    \end{equation}    
\end{lemma}
\begin{proof}
    First note that there exists $c_0$ such that 
    for each  $x\in \mathcal C_N$  the ball 
    $B(x,c_0\delta(x))$ lies inside $\mathcal H_N \cup\widetilde {\mathcal T}_N$. 
    Hence, the extension of $v_n(x)$ 
    together with its derivatives 
    is harmonic over $B(x,c_0\delta(x))$.

    Applying the mean-value formula for harmonic functions to function
  $\frac{\partial v_N(x)}{\partial x}$ and obtain
  \begin{align*}
    \left|\frac{\partial v_N(x)}{\partial x}\right|&= 
    \left|\frac{1}{Vol(B(x,c_0 \delta(x)))}
    \int_{B(x,c_0 \delta(x))} \frac{\partial v_N(y)}{\partial y}dy  \right|\\
    &= \left| \frac{1}{(c_0\delta(y))^d\alpha(d)}
    \int_{\partial B(x,c_0 \delta(x))} v_N\nu_i ds \right|\\
    &\le \frac{d\alpha(d)(c_0\delta(y))^{d-1}}{\alpha(d)(c_0\delta(y))^d} 
    \max_{y\in \partial B(y,c_0\delta_0(y))} v_N(y)\\
    &\le N \frac{d}{\delta(x)},
  \end{align*}
  where we used the fact that $v_N(x)\le N$.  
  As the classical maximum principle is 
  not directly 
  applicable ($v_N$ has a discontinuity at the vertex),  
  we  recall the representation of $v_N$ via the Poisson kernel 
  for the half plane and~\cite[Theorem 1.7.5]{AG01} 
  to establish that $v_N(x)\le N$.  
  Here $\alpha(d)$ is the volume of the unit ball and we used the
  Gauss-Green theorem. In the second line of the display $\nu_i$ is the outer
normal and integration takes place on the surface of the ball $B(x,\delta(y))$. 
The higher derivatives can be treated likewise. 
The claim of the Lemma immediately follows.
\end{proof}     

We will return to the random walk. 
Let \(v_{N}\) be the solution of~\eqref{eq:laplace} in  
\(\mathcal T_{N}\). 
Next we can estimate the error similarly to Lemma~8 of~\cite{DW15}. 
Let $f_N(x):=\E v_N(x+X)-v_N(x)$.  
\begin{lemma}\label{lem:error}
    There exists a constant $C$ such that for 
    $x\in \mathcal T_N\cap T\mathbb Z^2$,  
    \begin{equation}\label{eq:error}
        |f_N(x)| \le \frac{CN}{\delta(x)^6}. 
    \end{equation}    
\end{lemma}    
\begin{proof}
    Recall that $S(n)=TZ(n)$, that is 
 \[
    S_{1}(n) = \sqrt 2  Y_1(n) + \frac{1}{\sqrt 2} Y_2(n), \quad     
 S_{2}(n) = \sqrt{\frac{3}{2}}Y_2(n),  
 \quad n=0, 1,2,\ldots
 \]
 We will compute now first 5 moments of 
 $(X_1,X_2)$. The details of computation 
 are given in Appendix~\ref{sec:computation.of.moments}. 
 We have, for positive integers $n$ and $m$, 
 \begin{align*}
\E[X_1^{2n+1}]&=\E[X_2^{2n+1}]=0\\ 
 \E[X_2^2]&= \E[X_1^2]=1,\E[X_1X_2]=0  
 \\ 
 \E[X_1^{2n} X_2^{2m-1}]
    &=
    \E[X_1^{2n-1} X_2^{2m}]=0 \\ 
    \E[X_1^3 X_2]&=\E[X_1 X_2^3]=0\\
    \E[X_1^2 X_2^2]&=\frac{1}{2}, 
    \E[X_2^4]=\E[X_1^4]=\frac32.     
\end{align*}
Similarly to Lemma~8 of~\cite{DW15} we first write down the Taylor expansion 
    up to the 6th term. 
    Now these moments allow us to write down the first 
    five terms of the Taylor expansion, 
    \begin{align*}
        &\E[v_N(x_1+X_1,x_2+X_2)]=
        u(x)
        +\frac{1}{2}\Delta v_N(x_1,x_2)\\ 
        &\hspace{0.5cm}
        +\frac{1}{24}
        \left(
            (v_N)_{x_1^4}(x_1,x_2)\E[X_1^4]
            +6 (v_N)_{x_1^2x_2^2}(x_1,x_2)\E[X_1^2X_2^2]
            +(v_N)_{x_2^4}(x_1,x_2)\E[X_2^4]
        \right)+R_6(x)\\
        &\hspace{0.5cm}=
        v_N(x)+\frac{1}{2}\Delta v_N(x_1,x_2)+
        \frac{1}{24}
        \left(
            \frac 32 (v_N)_{x_1^4}(x_1,x_2) 
              +3 (v_N)_{x_1^2x_2^2}(x_1,x_2) 
              \frac 32 (v_N)_{x_2^4} (x_1,x_2)
        \right)    
        +R_6(x).  
    \end{align*}
    Now note  that since $v_N$ is  harmonic
    \[
    \Delta v_N=0    
    \]
    and 
    \[
        \frac 32 (v_N)_{x_1^4} 
        +3 (v_N)_{x_1^2x_2^2} 
        \frac 32 (v_N)_{x_2^4}
        =
        \frac 32 \Delta (v_N)_{x_1x_1}
        + \frac 32 \Delta (v_N)_{x_2x_2}=0
    \]
    Hence, only the  terms starting from the sixth matter and 
    \[
        \left|\E[(v_N)(x_1+X_1,x_2+X_2)]-u(x)\right| 
        =|R_6(x)| 
        \le C\max_{z \in C(x), |\alpha|=6}
        \left|\frac{\partial^\alpha}{\partial z^\alpha}(v_N)(z)\right|,
    \] 
    where $C(x)$ is a convex hull of points achievable from 
    $x$ in one jump. 
    Applying the estimates for 
    sixth partial 
    derivatives proved in Lemma~\ref{lem:derivatives} we arrive at the conclusion.     
\end{proof}

Recall that 
$P_{x_1,x_2}^{bm,(321)}(N)=\frac{v_N(x)}{N}$, where 
  $v_N(x)$ 
solves the Dirichlet problem~\eqref{eq:laplace} with 
$\phi(x)=
1-\sqrt{\frac{2}{3}}\frac{x_2}{N}$. 
Then, 
\[
    NP_{x_1,x_2}^{(321)}(N)=:V_N(x)=  
    \E_x[\phi(S_{\sigma_N})] 
    =\E_x v_{N}(S_{\sigma_N}). 
\]
Next note that 
\[
\E_x v_{N}(S_{\sigma_N}) - 
v_{N}(x) = \E_x \sum_{n=0}^{\sigma_N-1} f_N(S_k) = 
\sum_{y\in \mathcal T_N\cap T\mathbb Z^2} 
G_N(x,y)f_N(y). 
\]
and, therefore, 
\begin{equation}\label{eq.series} 
    V_N(x)=   
    \E_x v_{N}(S_{\sigma_N})
    =v_{N}(x) + 
    \sum_{y\in \mathcal T_N\cap T\mathbb Z^2} 
G_N(x,y)f_N(y).
\end{equation}
\begin{proof}[Proof of Proposition~\ref{prop.rate}] 
    We will  bound 
    the series in~\eqref{eq.series} to obtain the result. 
    First note that by reversing time 
    we obtain the following estimate 
    from~\eqref{eq:gn.local.bound},
    \[
      G_N(x,y)\le C\frac{V(y)}{N^3}.  
    \] 
    By symmetry this bound holds near all three vertices and hence we can write, 
    \[
      G_N(x,y)\le C \frac{\delta(y)^3}{N^3}.  
    \]
    Then, using Lemma~\ref{lem:error} we estimate 
    \[
        \sum_{y\in \mathcal T_N\cap T\mathbb Z^2} 
            G_N(x,y)|f_N(y)| 
            \le 
        C_\delta
        \sum_{y\in \mathcal T_N\cap T\mathbb Z^2} 
        \frac{\delta(y)^3}{N^3}
        \frac{N}{\delta(y)^6}
        \le \frac{C}{N^2},
    \] 
    as the series $\sum_{y\in \mathcal T_N\cap T\mathbb Z^2} 
    \frac{1}{\delta(y)^3}$ converges. 
    Therefore, 
    \[
        \left|
        P_{x_1,x_2}^{(321)}(N)-P_{x_1,x_2}^{bm,(321)}(N) 
        \right|=
    \left|\frac{V_N(x) - 
    v_{N}(x)}{N}\right| \le \frac{C}{N^3},    
    \]
    as required. 
\end{proof}


\appendix 
\section{Computation of moments}\label{sec:computation.of.moments} 
First  we will write down the  moments  of $(Y_1,Y_2)$.
We have, for positive integer $n$
\begin{align*}
    \E[(Y_1)^{2n}]= \E[(Y_2)^{2n}] =\frac{2}{3}\\ 
    \E[(Y_1)^{2n-1}]= \E[(Y_2)^{2n-1}] =0. 
\end{align*}    
Also, for all positive integers $n$ and $m$ 
\begin{align} 
\E[Y_{1}^{2n}Y_2^{2m-1}]&=\E[Y_{1}^{2n-1}Y_2^{2m}] = 0  \label{2n.2m-1}  \\ 
\E[Y_{1}^{2n-1}Y_2^{2m-1}]&= - \frac{1}{3}\label{2n-1.2m-1}\\ 
\E[(Y_1)^{2n}(Y_2)^{2m}] &= \frac{1}{3}\label{2n.2m}.    
\end{align}
Then, it follows from~\eqref{2n.2m-1} 
that odd moments disappear, 
\[
\E[X_1^{2n+1}]=\E[X_2^{2n+1}]=0.    
\]
Even moments are given by 
\[
  \E[X_2^2]=\frac 32 \E[Y_2^2]=1\quad  
  \E[X_2^4]=\frac 94 \E[Y_2^4]=\frac32   
\]
and 
\begin{align*}
    \E[X_1^2] &= \frac{1}{2} 
    \E[(2Y_1+Y_2)^2] =    
    \frac{1}{2}\left(
        4\frac{2}{3}-4\frac{1}{3}+\frac{2}{3}
    \right)=1,\\
    \E[X_1^4] &= \frac{1}{4} 
    \E[(2Y_1+Y_2)^4] =    
    \frac{1}{4}\left(
        16\cdot \frac{2}{3}
        -4\cdot 8\frac{1}{3}+
        6\cdot 4\frac{1}{3}
        -4\cdot 2\frac{1}{3}
        +\frac{2}{3}
    \right)=\frac{3}{2}.
\end{align*}    
Mixed moments are given by 
\begin{align*}
    \E[X_1X_2] &= \frac{\sqrt 3}{2}
    \E[Y_2^2+2Y_1Y_2] =0\\ 
    \E[X_1^2 X_2]&=
    \frac{\sqrt 3}{2\sqrt{2}}
    \E[Y_2^3+4Y_2^2Y_1+4Y_2Y_1^2]=0.
\end{align*} 
Similarly to the latter expression one can 
use~\eqref{2n.2m-1} to 
show 
for any positive integers $n$ and $m$ that  
\[ 
    \E[X_1^{2n} X_2^{2m-1}]
    =
    \E[X_1^{2n-1} X_2^{2m}]=0. 
\] 
Next 
\begin{align*}
\E[X_1^3 X_2]&
=
\frac{\sqrt 3}{4} 
    \E[(Y_2+2Y_1)^3Y_2]\\ 
&= 
\frac{\sqrt 3}{4} 
\left(
    \E[Y_2^4+6\E[Y_2^3Y_1]+12\E[Y_2^2Y_1^2]
    +8\E[Y_2Y_1^3]]
\right)\\
&=
\frac{\sqrt 3}{4} 
\left(
    \frac 23 -6\cdot \frac13+12\cdot\frac 13 -8\cdot\frac13
    \right)=0    
\end{align*}
and 
\begin{align*}
    \E[X_1 X_2^3]&
    =
    \frac{3\sqrt 3}{4} 
        \E[(Y_2+2Y_1)Y_2^3]\\ 
    &= 
    \frac{3\sqrt 3}{4} \left(
        \E[Y_2^4]+2\E[Y_2^3Y_1]\right)
        =
        \frac{3\sqrt 3}{4}
        \left(2/3-2/3
            \right)=0.
\end{align*}
Finally, 
\begin{align*}
\E[X_1^2 X_2^2]&
    =
    \frac{3}{4} 
        \E[(Y_2+2Y_1)^2Y_2^2]\\ 
    &= 
    \frac{3}{4}  \left(
        \E[Y_2^4]+4\E[Y_2^3Y_1]
        +4\E[Y_2^2Y_1^2
        ]\right)\\ 
        &= 
        \frac{3}{4}  \left(
            \frac 23 -4\cdot \frac 13  +4\cdot \frac 13
            \right)      =\frac{1}{2}.
\end{align*}



    \end{document}